\documentclass[journal,twoside]{IEEEtran}

\usepackage{amsmath,amssymb}

\DeclareSymbolFont{bbold}{U}{bbold}{m}{n}
\DeclareSymbolFontAlphabet{\mathbbold}{bbold}

\usepackage{amsthm}
\theoremstyle{remark}
\newtheorem{thm}{Theorem}
\theoremstyle{remark}
\newtheorem{lem}{Lemma}
\theoremstyle{remark}

\theoremstyle{remark}
\newtheorem{assum}{Assumption}
\theoremstyle{remark}

\theoremstyle{remark}
\newtheorem*{rem}{Remark}
\theoremstyle{remark}
\newtheorem{defi}{Definition}
\theoremstyle{remark}
\newtheorem{pro}{Proposition}

\usepackage{accents}
\newcommand{\ubar}[1]{\underaccent{\bar}{#1}}

\DeclareMathOperator{\inte}{int}
\DeclareMathOperator{\cl}{cl}
\DeclareMathOperator{\re}{Re}
\DeclareMathOperator{\im}{Im}

\usepackage{nomencl}
\makenomenclature

\usepackage{bm}
\newcommand{\bg}[1]{\boldsymbol{\mathbf{#1}}}

\usepackage[dvipsnames]{xcolor}
\usepackage[colorlinks,linkcolor=blue,citecolor=ForestGreen]{hyperref}

\usepackage{graphicx,algpseudocode}
\usepackage[boxed]{algorithm2e}

\usepackage{pgfplots}
\pgfplotsset{compat=1.14}

\allowdisplaybreaks

\usepackage{cite}

\usepackage{tikz}

\usepackage{multirow}

\usepackage{booktabs}

\newenvironment{customlegend}[1][]{%
\begingroup
\csname pgfplots@init@cleared@structures\endcsname
\pgfplotsset{#1}%
}{%
\csname pgfplots@createlegend\endcsname
\endgroup
}%
\def\addlegendimage{\csname pgfplots@addlegendimage\endcsname}

\begin{document}
%
\title{Novel Region of Attraction Characterization for Control and Stabilization of Voltage Dynamics}
%
%
%

\author{Bai~Cui,~\IEEEmembership{Member,~IEEE,}
    Ahmed~Zamzam,~\IEEEmembership{Member,~IEEE,}
    Guido~Cavraro,~\IEEEmembership{Member,~IEEE,}
    and~Andrey~Bernstein,~\IEEEmembership{Senior~Member,~IEEE}
    }
\maketitle

\begin{abstract}
In this paper, we study the monitoring and control of long-term voltage stability considering load tap-changer (LTC) dynamics. We show that under generic conditions, the LTC dynamics always admit a unique stable equilibrium. For the stable equilibrium, we characterize an explicit inner approximation of its region of attraction (ROA). Compared to existing results, the computational complexity of the ROA characterization is drastically reduced. A quadratically constrained linear program formulation for the ROA characterization problem is proposed. In addition, we formulate a second-order cone program for online voltage stability monitoring and control exploiting the proposed ROA characterization, along with an ADMM-based distributed algorithm to solve the problem. The efficacy of the proposed formulations and algorithms is demonstrated using a standard IEEE test system.
\end{abstract}

\begin{IEEEkeywords}
Voltage stability, load tap changer, region of attraction, distributed optimization.
\end{IEEEkeywords}

%
\IEEEpeerreviewmaketitle

\section{Introduction}
\IEEEPARstart{P}{ower} systems have traditionally been designed with sufficient margins against disturbances and contingencies. However, with ever increasing power demand and competitive electricity market, they are being operated closer to the operational boundaries, in other words, their loading margins to the operational boundaries are shrinking. Systems with insufficient loading margins run the risk of resulting in catastrophic outcomes such as cascading failure and large-scale blackouts. Voltage collapse events have been the main culprits in several major blackouts worldwide.

The destabilizing behavior of load tap-changer (LTC) is one of the prime mechanisms of voltage collapse in bulk power systems. The post-disturbance secondary voltage restoration by LTC leads to load power restoration, which may exacerbate the already impaired post-disturbance transfer capability and accelerate long-term voltage collapse. Traditionally, long-term voltage instability is generally modeled by saddle-node bifurcation of the underlying quasi-steady-state model. It is therefore customary in long-term voltage stability analysis to adopt steady-state power flow model with constant power loads and examine conditions associated with the singularity of power flow Jacobian. Most online voltage stability indices are derived based on this model \cite{kessel1986estimating, tiranuchit1988posturing, glavic2011short, wang2011voltage, wang2017necessary, cui2017voltage}. While constant power load model captures the stability margin and its sensitivities to system parameters when the system loses its stability through the loss of equilibrium of the long-term dynamics \cite{dobson2011irrelevance}, it is not capable of modeling the other common instability mechanism: the instability through a lack of attraction towards the stable long-term equilibrium \cite[Sect. 8.2.2]{van2007voltage}. To capture both mechanisms, load model with explicit modeling of LTC dynamics should be employed. In this work, we study the long-term voltage stability problem incorporating LTC dynamics. 

Long-term voltage stability of networked LTCs has been studied in the seminal work \cite{liu1989analysis} using a continuous-time approximation model where the stability of equilibria has been identified and the region of attraction (ROA) of the stable equilibrium was characterized. Further, characterization of the ROA for discrete-time LTC model appears in \cite{vournas2006region}. However, explicit characterization is only given for radial three-bus systems excluding mesh structures that are evident in the vast majority of transmission networks.

To restore a long-term equilibrium after system disturbance, emergency voltage stability control is needed, which can be done at generation side (generator terminal voltage boosting), transmission side (reactive power compensation), or load side (load reduction) \cite[Sect. 8.6]{van2007voltage}. Emergency control countermeasures at both the generation and transmission sides assume the availability of certain reactive power reserve, which may be nonexistent in emergency scenario with degraded system condition. On the other hand, load control can be realized either through direct load shedding, or through emergency LTC controls including tap blocking, set-point reduction, and tap reversing \cite{vournas2004load}, all of which cause power/voltage quality degradation for end-users.

The proliferation of distributed energy resources (DER) in distribution system provides an alternative, where recent advancement in control technologies has made real-time coordination of massive amount of DERs a reality \cite{bernstein2019real}. With proper coordination, the aggregated DERs are able to provide various transmission-level services such as frequency regulation and voltage support \cite{arnold2018model, dall2018optimal, valverde2019coordination}. It is hypothesized in \cite{singhal2018framework} that they can be viewed as `mini-static var compensator' and provide local reactive power support to the grid by compensating reactive power demand. In fact, there is plausible reason for DERs to participate in emergency voltage stability control: a case study has shown that the DER controllers can contribute to bulk power system voltage collapse if they are agnostic about the emergency condition and not properly controlled under such circumstances \cite{aristidou2017contribution}. A similar approach of coordinated distributed load shedding and LTC control has been shown to be effective in emergency control \cite{otomega2014two}. In this work, we assume the availability of distributed DER controllers which command required reactive power support at the secondary level and focus on the determination of optimal aggregated reactive power support at the transmission system level.

In this paper, we analyze the stability properties of networked LTCs including the stability of equilibria and explicit characterization of ROA of the stable equilibrium. The novelty of the paper is two-fold. First, two new technical results concerning the stability of networked LTC system are given: 1) We show that the stable equilibrium is unique generically; and 2) we provide a novel ROA characterization of the stable equilibrium which can be efficiently computed.
Both conditions are built on and improve those in \cite{liu1989analysis}. Second, based on the technical results, a distributed optimization algorithm is proposed for stability monitoring and emergency reactive power support computation.

The remainder of the paper is organized as follows. The notations, system modeling, and problem statement are given in Section \ref{sect:notations}. Novel stability results, i.e., uniqueness of stable equilibrium and characterization of its ROA are presented in Section \ref{sect:ROA}. Optimization formulation of the stability monitoring and control problem as well as its distributed implementation are discussed in Section \ref{sect:opt}. Finally, case studies coroborating the efficacy of the proposed approaches are demonstrated in Section \ref{sect:sim}.

\section{Notation and System Modeling} \label{sect:notations}

\subsection{Notation}

The set of real, positive real, and complex numbers are denoted by $\mathbb{R}$, $\mathbb{R}_{>0}$, and $\mathbb{C}$, respectively. Scripted capital letters $\mathcal{A}, \mathcal{B}, \ldots$ are used to represent other sets. The interior, closure, and boundary of a set $\mathcal{A}$ are denoted by $\inte(\mathcal{A})$, $\cl(\mathcal{A})$, and $\partial \mathcal{A}$, respectively. 
Vectors and matrices are represented by boldface letters while scalars are represented by normal ones.
For matrix $\bg{A} \in \mathbb{R}^{m\times n}$, $\bg{A}^\top$ is the transpose of $\bg{A}$. $\bg{a}_i$ denotes the vector formed by the $i$th row of $\bg{A}$ and $\bg{A}_i$ denotes the vector formed by the $i$th column of $\bg{A}$. For real square matrices $\bg{M}, \bg{N}$, the expression $\bg{M} > 0$ (resp. $\bg{M} > \bg{N}$) means $\bg{M}$ (resp. $\bg{M} - \bg{N}$) is an element-wise positive matrix. For vector $\bg{x} \in \mathbb{R}^n$, $\| \bg{x} \|_p$ denotes the $\ell_p$ norm of $\bg{x}$ where $p \in [1, \infty) \cup \{\infty\}$ and $[\bg{x}] \in \mathbb{R}^{n \times n}$ denotes the associated diagonal matrix. The $n$-dimensional open ball of $\|\cdot\|_{p}$ centered at $\bg{c}$ with radius $r$ is $\mathcal{B}_p(\bg{c}, r) = \{ \bg{x}\in\mathbb{R}^n :\; \| \bg{x} - \bg{c} \|_p < r \}$. $\mathbbold{0}$ and $\mathbbold{1}$ are the vectors of all $0$'s and $1$'s of appropriate sizes. The cardinality of a set or the absolute value of a (possibly) complex number is denoted by $|\cdot|$. $j = \sqrt{-1}$ is the imaginary unit. $\re(\cdot)$ and $\im(\cdot)$ denote the real and imaginary part of a complex number. 

\subsection{System Modeling}

We consider a connected and phase-balanced power system with $n+m$ buses operating in steady-state. The underlying topology of the system can be described by an undirected connected graph $\mathcal{G} = (\mathcal{V}, \mathcal{E})$, where buses are modeled as vertices $\mathcal{V}$ and lines are represented by edges $\mathcal{E} \subseteq \mathcal{V} \times \mathcal{V}$. The buses are categorized into two mutually exclusive sets: generators ($\mathcal{V}_G$) and loads ($\mathcal{V}_L$), such that $\mathcal{V}_G \cup \mathcal{V}_L = \mathcal{V}$ and $\mathcal{V}_G \cap \mathcal{V}_L = \emptyset$. The load buses are numbered from $1$ to $n$ and the generator buses are numbered from $n+1$ to $n+m$, i.e., $\mathcal{V}_L = \{1, \ldots, n\}$ and $\mathcal{V}_G = \{n+1, \ldots, n+m\}$. Every bus $i$ in the system has voltage $V_i$. We assume the induced subgraph of $\mathcal{G}$ with vertex set $\mathcal{V}_L$ is connected, i.e., removing vertex set $\mathcal{V}_G$ and the incident lines from $\mathcal{G}$ does not disconnect $\mathcal{G}$.

Each load bus $i$ is modeled as a constant admittance behind an LTC with tap ratio $r_i : 1$ where $r_i > 0$ is normally around 1. The voltage on the secondary side of the LTC is denoted by $V_{s,i}$, so we have $V_i / V_{s,i} = r_i$. The admittance at the secondary side of the LTC is $-jb_{s,i}$. We assume the loads are inductive so that $b_{s,i} > 0$ for all $i \in \mathcal{V}_L$. When $r_i$ is fixed, load bus $i$ is equivalent to a constant admittance bus with admittance $-jb_i$ where $b_i = b_{s,i} / r_i^2$. We adopt the continuous-time dynamics approximation in \cite[Sect. 4.4]{van2007voltage} and model LTC dynamic at load bus $i$ is
\begin{align} \label{eq:dynbusi}
    \dot{r}_i = \frac{1}{T_i} (V_{s,i} - V_{0,i}).
\end{align}
On the other hand, generator buses are modeled as constant voltage sources with fixed $V_i$ for all $i \in \mathcal{V}_G$.

Since transmission systems have negligible $r/x$ ratio, we assume the network is lossless so that line $(i,k)$ has admittance $y_{ik} = -jb_{ik}$ where the line susceptance $b_{ik}$ is positive. Let $b_{ik} = 0$ for $(i,k) \notin \mathcal{E}$. The bus admittance matrix is defined as $-j\bg{B}$, where the susceptance matrix $\bg{B} \in \mathbb{R}^{(n+m)\times (n+m)}$ is
\begin{align}
    B_{ik} = 
    \begin{cases}
        -b_{ik}, & i \ne k  \\
        \sum_{k=1}^{n+m} b_{ik}, & i=k \text{ and } i > n \\
        b_i + \sum_{k=1}^{n+m} b_{ik}, & i=k \text{ and } i \le n
    \end{cases}
\end{align}

If we partition the susceptance matrix based on load and generator buses as
\begin{align}
    \bg{B} = \begin{bmatrix} \bg{B}_{LL} & \bg{B}_{LG} \\ \bg{B}_{GL} & \bg{B}_{GG} \end{bmatrix},
\end{align}
then based on Ohm's and Kirchhoff's law, the load voltages are determined by
\begin{align} \label{eq:h}
    \bg{B}_{LL} \bg{V}_L = \bg{h},
\end{align}
where $\bg{h} = -\bg{B}_{LG}\bg{V}_G \ge \mathbbold{0}$; $\bg{V}_L$ collects the load voltages; and $\bg{V}_G$ collects the generator voltages. Note that $\bg{h}$ is a constant vector while the diagonal elements of $\bg{B}_{LL}$ depend on $\bg{r}$ (the vector of tap ratios). By definition, $-\bg{B}_{LL}$ is an irreducible Hurwitz M-matrix, so $\bg{Z} := \bg{B}_{LL}^{-1} > 0$, i.e., $\bg{Z}$ is element-wise positive.

Given tap ratios, we can solve for the voltages at load bus $i$ as
\begin{align} \label{eq:V=zh}
    V_i(\bg{r}) = \bg{z}_i^\top \bg{h}, \quad V_{s,i}(\bg{r}) = \frac{V_i}{r_i} = \frac{\bg{z}_i^\top \bg{h}}{r_i},
\end{align}
where $\bg{z}_i^\top$ is the $i$th row of the impedance matrix $\bg{Z}$. 

In this paper, we study the following dynamical system of LTCs:
\begin{align} \label{eq:dynamics}
    \dot{r}_i = \frac{1}{T_i} (V_{s,i}(\bg{r}) - V_{0,i}), \quad \forall i \in \mathcal{V}_L.
\end{align}

Let the set of positive equilibria of \eqref{eq:dynamics} be $\mathcal{M}$, which are solutions to the system of algebraic equations $\dot{r}_i = 0, \forall i \in \mathcal{V}_L$:
\begin{align} \label{eq:M}
    \mathcal{M} = \left\{ \bg{r} \in \mathbb{R}^n_{>0} :\; \dot{r}_i(\bg{r}) = 0, \; \forall i \in \mathcal{V}_L \right\}.
\end{align}

\begin{defi}
    Let $\phi(t;\bg{r})$ be the solution of \eqref{eq:dynamics} that starts at initial state $\bg{r}$ at time $t = 0$ and let $\bg{r}^*$ be an equilibrium that is asymptotically stable. The ROA of $\bg{r}^*$ is defined as a positively invariant set under \eqref{eq:dynamics} such that $\lim_{t \to \infty} \phi(t; \bg{r}) = \bg{r}^*$.
\end{defi}

\subsection{Problem Statement}

We are interested in stability monitoring and control of a network hosting multiple interacting LTCs. Specifically, we address the following two problems in the paper:
\begin{itemize}
    \item For a given network with fixed secondary voltage set-points and load admittance, develop a computationally-efficient characterization of the ROA of the stable equilibrium of \eqref{eq:dynamics}.
    
    \item For a given network with fixed secondary voltage set-points and tap position, develop an efficient algorithm to determine the minimum secondary support (in terms of reduction in load admittance $\bg{b}_s$) such that the tap position lies in the ROA of the stable equilibrium of \eqref{eq:dynamics}.
\end{itemize}
The two problems are addressed in Sections \ref{sect:ROA} and \ref{sect:opt}, respectively.

\section{Characterization of ROA} \label{sect:ROA}



We first note that there is a one-to-one correspondence between equilibria of \eqref{eq:dynamics} and power flow solutions of a corresponding set of power flow equations:
\begin{pro}
    There is a one-to-one correspondence between elements in the set $\mathcal{M}$ defined in \eqref{eq:M} and power flow solutions of a corresponding set of power flow equations.
\end{pro}
\begin{proof}
    The point $\bg{r}^* \in \mathbb{R}_+^{n}$ is an equilibrium of \eqref{eq:dynamics} if and only if the reactive power demand is $Q_i^* = V_{0,i}^2b_{s,i}$ for all $i \in \mathcal{V}_L$. Said differently, $\bg{r}^*$ is an equilibrium of \eqref{eq:dynamics} if and only if $\bg{V}_L^* = [\bg{r}^*]\bg{V}_{0,L}$ is a power flow solution to the following power flow equations:
    \begin{align} \label{eq:pf}
        [\bg{V}_L] \left( \tilde{\bg{B}}_{LL}\bg{V}_L + \bg{B}_{LG}\bg{V}_G \right) = -\bg{Q}_L^*,
    \end{align}
    where $\tilde{\bg{B}}$ is the load susceptance submatrix $\bg{B}_{LL}$ in which all $b_i = 0$ (i.e., the load susceptance is absent in the susceptance matrix). 
\end{proof}


We make the following assumption on the set $\mathcal{M}$:
\begin{assum} \label{assum:isolated}
    The set $\mathcal{M}$ is a discrete set. 
\end{assum}
The assumption implies the set of power flow equations corresponding to the algebraic equations describing $\mathcal{M}$ are generic so that they do not share common components. B\'ezout's theorem then ensures that the number of equilibria are finite and they are isolated, i.e., every equilibrium is unique in a sufficiently small neighborhood. This is a very mild technical assumption which holds almost always. Moreover, it has been shown in \cite{baillieul1982geometric} that for a more simplified power flow model where bus voltages are fixed, the power flow solution set is composed of finite number of isolated points with measure one on the set of system parameter.

We define the set of tap positions whose corresponding secondary voltages are higher than their set-points by $\mathcal{P}$:
\begin{defi}
    The set $\mathcal{P}$ is defined as
    \begin{align} \label{eq:P}
        \mathcal{P} = \left\{ \bg{r} \in \mathbb{R}^n_{>0} :\; \dot{r}_i(\bg{r}) \ge 0, \; \forall i \in \mathcal{V}_L \right\}.
    \end{align}
\end{defi}
Note that $\mathcal{P}$ contains all the equilibria, that is, $\mathcal{M} \subseteq \mathcal{P}$.

Some known results regarding the dynamical system \eqref{eq:dynamics} are presented in Appendix \ref{app:known}. The theorems certify the existence of a maximum equilibrium $\bg{\alpha} \in \mathbb{R}^n_{> 0}$ such that $\bg{\alpha} \ge \bg{r}^*$ for all $\bg{r}^* \in \mathcal{M}$. In addition, $\bg{\alpha}$ is asymptotically stable as long as the Jacobian of \eqref{eq:dynamics} is nonsingular at $\bg{\alpha}$. The same result for power flow equations (i.e., the existence of a stable high-voltage power flow solution) has been obtained in recent papers \cite{dvijotham2017high, simpson2016voltage}, along with algorithm that provably finds the solution.

We are interested in explicitly characterizing the ROA of $\bg{\alpha}$. Our characterization improves upon the existing one in \cite{liu1989analysis}, which is slightly paraphrased as below:
\begin{thm}[\hspace{1sp}{\cite[Prop. 3]{liu1989analysis}}]
    \label{thm:liu}
    For any $\ubar{\bg{r}} \in \mathbb{R}^n_{>0}$,  the set $\mathcal{A}(\ubar{\bg{r}}) := \{ \bg{r}: \bg{r} \ge \ubar{\bg{r}} \}$ is an ROA of $\bg{\alpha}$ if: (i)  $\bg{V}_s(\ubar{\bg{r}}) \ge \bg{V}_0$ and (ii) $\bg{\alpha}$ is the only equilibrium in $\mathcal{A}(\ubar{\bg{r}})$. 
\end{thm}


The ROA characterization provided in Theorem \ref{thm:liu} is implicit and the main obstacle to directly apply the theorem lies in certifying the non-existence of equilibria other than $\bg{\alpha}$ in $\mathcal{A}(\ubar{\bg{r}})$. While computational algebraic geometry approaches exist to locate all power flow solutions (which, as we have noted, is equivalent to locating all equilibria of \eqref{eq:dynamics}), they are computationally intensive and are not scalable to systems of realistic size \cite{mehta2016numerical}. However, the following result shows no extra effort is needed to check equilibrium uniqueness: 
\begin{lem} \label{thm:uniqueeq}
   There is a unique equilibrium of \eqref{eq:dynamics} in $\mathcal{A}(\bg{r})$ (other than possibly $\bg{r}$ itself) for any $\bg{r} \in \mathcal{P} \setminus \{\bg{\alpha}\}$.
\end{lem}
\begin{proof}
    See Appendix \ref{app:uniqueeq}.
\end{proof}

By combining Theorem \ref{thm:liu} and Lemma \ref{thm:uniqueeq}, we obtain the following characterization of ROA which is computationally attractive:
\begin{thm} \label{thm:ROA}
    The set $\mathcal{A}(\ubar{\bg{r}}) := \{ \bg{r}: \bg{r} \ge \ubar{\bg{r}} \}$ is an ROA of $\bg{\alpha}$ for any $\ubar{\bg{r}} \in \mathcal{P}$.
\end{thm}
The theorem implies that the only information needed to characterize an ROA is the availability of a point in $\mathcal{P}$. We address the problem of finding such a point in the next section.

The next theorem shows that among all possible equilibria, it suffices to study $\bg{\alpha}$ alone, since it is the only stable equilibrium:
\begin{thm} \label{thm:eq_unstable}
   All equilibria of \eqref{eq:dynamics} other than $\bg{\alpha}$ are unstable.
\end{thm}
\begin{proof}
    See Appendix \ref{app:eq_unstable}.
\end{proof}

\begin{rem}
The equilibria are classified into ``high tap ratio'' ones and others in \cite{liu1989analysis}, and it is shown therein that only ``high tap ratio'' equilibria, $\bg{\alpha}$ included, can be stable. While it makes engineering sense to focus on $\bg{\alpha}$ since it corresponds to the stable high-voltage operating point at which normal system is operated, the result in \cite{liu1989analysis} does not preclude the possibility of the existence of other stable equilibria. With the mild technical Assumption \ref{assum:isolated}, we establish that $\bg{\alpha}$ is in fact the only stable equilibrium. Therefore, we can be certain that this is the equilibrium that a nominal system would operate around.
\end{rem}

\section{Stability Monitoring and Instability Mitigation} \label{sect:opt}

In this section, we address the problems of computing ROA and stability monitoring \& control by formulating two efficient optimization problems, leveraging the analytical characterization of ROA developed in Section \ref{sect:ROA}. We show that the ROA computation problem can be formulated as a quadratically constrained linear program, whereas the stability problem of determining the minimum secondary support to restore system stability admits a safe second-order cone program (SOCP) approximation, whose approximation quality will be validated through numerical experiments in Section \ref{sect:sim}.


\subsection{Characterization of ROA}

Given network with fixed secondary voltage setpoints and load admittance, the problem of finding a tap position $\ubar{\bg{r}} \in \mathcal{P}$ to characterize the ROA of $\bg{\alpha}$ based on Theorem \ref{thm:ROA} can be formulated as follows:
\begin{subequations} \label{eq:p1:original}
\begin{align}
    \min_{\bg{V}, \bg{r}} \quad & \bg{c}^\top\bg{r} \\
    \text{s.t.} \quad & \left(\tilde{\bg{B}}_{LL} + [\bg{b}_s][\bg{r}]^{-2}\right)\bg{V} = \bg{h} \label{eq:p1:original:b} \\
    & \bg{V} \ge [\bg{r}] \bg{V}_{0} \label{eq:p1:original:c} \\
    & \bg{r} \ge \mathbbold{0}. \label{eq:p1:original:d}
\end{align}
\end{subequations}
Recall $\tilde{\bg{B}}_{LL}$ is the load susceptance matrix defined in \eqref{eq:pf}, $b_{s,i}$ is the positive load susceptance at bus $i \in \mathcal{V}_L$ and the `weighted generator voltage' vector $\bg{h}$ appeared in \eqref{eq:h}. To maximize the volume of the ROA, we design the objective function to find the minimum $\bg{r} \in \mathcal{P}$ along some direction $\bg{c} \ge \mathbbold{0}$ per Theorem \ref{thm:ROA}. Constraint \eqref{eq:p1:original:b} enforces Ohm's law and Kirchhoff's law over the network, while constraints \eqref{eq:p1:original:c} and \eqref{eq:p1:original:d} require that $\bg{r} \in \mathcal{P}$.

If we denote the optimal solution of problem \eqref{eq:p1:original} by $(\bg{V}^*(\bg{c}),\bg{r}^*(\bg{c}))$, then $\mathcal{A}(\bg{r}^*(\bg{c}))$ is an ROA. Each direction $\bg{c}$ determines a (possibly distinct) inner approximation of the true ROA, and their union therefore characterizes a maximal inner approximation of the ROA:
\begin{align}
    \mathcal{A}_{\cup} := \bigcup_{\substack{\bg{c} \ge \mathbbold{0} \\ \bg{c}^\top\mathbbold{1} = 1}} \mathcal{A}(\bg{r}^*(\bg{c})).
\end{align}
For practical implementation considerations, a few representative cost vectors $\bg{c}$ can be chosen based on specific system characteristics.

\subsection{Online Stability Monitoring and Control} \label{sect:stability_formulation}

In this section, we formulate a second-order cone program for the problem of stability monitoring and control. To certify that the tap position $\bg{r}_0$ is in the ROA of the stable equilibrium, it suffices to ensure that $\bg{r}_0 \in \mathcal{A}(\bg{r})$ for some $\bg{r} \in \mathcal{P}$, which is equivalent to achieving a zero cost for the following problem:
\begin{subequations}\label{eq:chk:f1}
\begin{align}
    \min_{\bg{r}, \bg{V}}\quad & \| \tilde{\bg{B}}_{LL} \bg{V} + [\bg{b}_s][\bg{r}]^{-2} \bg{V} - \bg{h} \|_2^2 \label{eq:chk:cost} \\
    \text{s.t.} \quad 
    & \bg{V} \ge [\bg{r}]\bg{V}_0 \label{eq:chk:c1}\\
    & \mathbbold{0} \le \bg{r} \le \bg{r}_0. \label{eq:chk:c2}
\end{align}
\end{subequations}

The above problem can be reformulated as an SOCP, and hence, it becomes amenable to distributed implementation with convergence guarantee by the standard ADMM algorithm \cite[Sect. 5.4]{bertsekas2015convex}. To see this, we first reformulate \eqref{eq:chk:f1} as
\begin{subequations}\label{eq:reform}
\begin{align}
    \min_{\bg{r}, \bg{V}}\quad & \| \tilde{\bg{B}}_{LL} \bg{V} + [\bg{b}_s][\bg{r}]^{-2} \bg{V} - \bg{h} \|_2^2 \label{eq:reform:cost} \\
    \text{s.t.} \quad 
    & [\bg{r}]^{-2}[\bg{V}]\bg{V} \ge [\bg{V}_0]\bg{V}_0 \label{eq:reform:c1}\\
    & [\bg{r}]^{-2}\bg{V} \ge [\bg{r}_0]^{-2}\bg{V} \label{eq:reform:c2} \\
    & \bg{V} \ge \mathbbold{0}. \label{eq:reform:c3}
\end{align}
\end{subequations}
Introducing a new variable $\bg{u} := [\bg{r}]^{-2} \bg{V}$, \eqref{eq:reform} can be reformulated as the following SOCP:
\begin{subequations} \label{eq:socp}
\begin{align}
    \min_{\bg{u}, \bg{V}}\quad & \| \tilde{\bg{B}}_{LL} \bg{V} + [\bg{b}_s]\bg{u} - \bg{h} \|_2^2 \label{eq:socp:cost} \\
    \text{s.t.} \quad 
    & [\bg{u}]\bg{V} \ge [\bg{V}_0]\bg{V}_0 \label{eq:socp:c1}\\
    & \bg{u} \ge [\bg{r}_0]^{-2}\bg{V} \label{eq:socp:c2} 
    \\
    & \bg{V} \ge \mathbbold{0}. \label{eq:socp:c3}
\end{align}
\end{subequations}
Problem \eqref{eq:socp} is an SOCP since it minimizes $\ell_2$-norm over linear constraints \eqref{eq:socp:c2}--\eqref{eq:socp:c3} and constraint \eqref{eq:socp:c1}, which is SOC representable as it can be written as $\| [V_{0,i}, (u_i - V_i)/2] \|_2 \le (u_i + V_i)/2, \; \forall i \in \mathcal{V}_L$.

A related problem is to determine the corrective actions to mitigate instability when $\bg{r}_0$ does not lie in the ROA. Given network with fixed secondary voltage setpoints and tap position $\bg{r}_0$, we want to determine the minimum reduction in load admittance $\bg{d}$ such that $\bg{r}_0$ returns to the ROA of the stable equilibrium. The minimum amount of aggregate secondary reactive power support (in terms of $\ell_2$-norm of reduction in secondary load susceptance) can be determined by the following optimization problem:
\begin{subequations} \label{eq:p2:original}
\begin{align}
    \min_{\bg{V}, \bg{r}, \bg{d}} \quad & \|\bg{d}\|_2^2 \\
    \text{s.t.} \quad & \left(\tilde{\bg{B}}_{LL} + [\bg{b}_s - \bg{d}][\bg{r}]^{-2}\right)\bg{V} = \bg{h} \label{eq:p2:original:b} \\
    & [\bg{r}]^{-1}\bg{V} \ge \bg{V}_{0} \\
    & \mathbbold{0} \le \bg{r} \le \bg{r}_0 \label{eq:p2:original:d} \\
    & \mathbbold{0} \le \bg{d} \le \bg{b}_s. \label{eq:p2:original:e}
\end{align}
\end{subequations}
Similar to \eqref{eq:socp}, problem \eqref{eq:p2:original} can be reformulated as
\begin{subequations} \label{eq:p2:reform}
\begin{align}
    \min_{\bg{V}, \bg{u}, \bg{d}} \quad & \|\bg{d}\|_2^2 \\
    \text{s.t.} \quad & \tilde{\bg{B}}_{LL} \bg{V} + [\bg{b}_s]\bg{u} - [\bg{d}]\bg{u} = \bg{h} \label{eq:p2:reform:b} \\
    & \eqref{eq:socp:c1}, \eqref{eq:socp:c2}, \eqref{eq:socp:c3}, \eqref{eq:p2:original:e}. \nonumber
\end{align}
\end{subequations}
However, the above problem is still nonconvex due to the bilinear term $d_i u_i$ in \eqref{eq:p2:reform:b}. To get a safe estimate of the minimum secondary support tractably, we formulate a convex surrogate of \eqref{eq:p2:reform} by treating $[\bg{d}]\bg{u}$ in \eqref{eq:p2:reform} as a single variable to be minimized. By \eqref{eq:p2:reform:b}, minimizing $\| [\bg{d}]\bg{u} \|_2$ is the same as minimizing $\| \tilde{\bg{B}}_{LL} \bg{V} + [\bg{b}_s]\bg{u} - \bg{h} \|_2$, we therefore arrive at the following convex program: 
\begin{subequations} \label{eq:p2:convex}
\begin{align}
    \min_{\bg{u}, \bg{V}} \quad & \| \tilde{\bg{B}}_{LL} \bg{V} + [\bg{b}_s]\bg{u} - \bg{h} \|_2^2 \\
    \text{s.t.} \quad & \tilde{\bg{B}}_{LL}\bg{V} \le \bg{h} \label{eq:p2:convex:c} \\
    & \eqref{eq:socp:c1}, \eqref{eq:socp:c2}, \eqref{eq:socp:c3}. \nonumber
\end{align}
\end{subequations}
The susceptance reduction $\bg{d}^*$ can be solved for from the optimal solution $(\bg{u}^*, \bg{V}^*)$ of the above problem as $\bg{d}^* = [\bg{u}^*]^{-1}(\tilde{\bg{B}}_{LL} \bg{V}^* + [\bg{b}_s]\bg{u}^* - \bg{h})$. This is a feasible solution to the original problem \eqref{eq:p2:original} and therefore provides an upper bound on the global minimum of \eqref{eq:p2:original}. In addition, $(\bg{u}^*,\bg{V}^*)$ is feasible for \eqref{eq:p2:convex} as long as $(\bg{u}^*,\bg{V}^*,\bg{d}^*)$ is feasible for \eqref{eq:p2:reform}. Since the original problem \eqref{eq:p2:original} is always feasible for sufficiently small $\bg{r}$ when setting $\bg{d} = \bg{b}_s$, the feasibility of the convex surrogate \eqref{eq:p2:convex} can be guaranteed.

Problems \eqref{eq:socp} and \eqref{eq:p2:convex} are the same except for the additional constraint \eqref{eq:p2:convex:c} in \eqref{eq:p2:convex}. We see that \eqref{eq:socp} achieves zero optimal cost if and only if the same holds true for \eqref{eq:p2:convex}. Therefore, solving problem \eqref{eq:p2:convex} serves dual purposes: when the optimal cost is zero, the tap position $\bg{r}_0$ is certified stable; otherwise, the optimal solution determines the amount of reactive power support to steer $\bg{r}_0$ back to ROA.


\subsection{Distributed Implementation} \label{sect:admm}

Solving the optimization problem~\eqref{eq:p2:convex} in an online fashion for a large-scale system requires network-wide knowledge of the tap positions and voltages. We opt for a distributed algorithm for solving this problem due to the following two reasons. First, utilities may not be willing to share local information to a central system operator due to privacy concerns; and second, a distributed solver can better adapt to time-varying system conditions and reject disturbances. 

To facilitate the design, suppose the underlying graph $\mathcal{G}$ of the system is partitioned into $n_s$ connected induced subgraphs (agents). Let $\mathcal{N}_i$ be the bus set of the $i$th agent, $\mathcal{N}_i^a$ be the set of buses adjacent to the $i$th agent, and $\mathcal{B}$ be the set of boundary buses, that is, buses with at least one adjacent bus in a different agent. Let $\bg{x}^i = \left( \{V_j\}_{j \in \mathcal{N}_i}, \{u_j\}_{j \in \mathcal{N}_i}, \{W^{ij}\}_{j \in \mathcal{N}_i^a} \right)$ collect the optimization variables of agent $i$. In particular, $\{V_j\}_{j \in \mathcal{N}_i}$ are voltages of agent $i$, $\{u_j\}_{j \in \mathcal{N}_i}$ are scaled voltages of agent $i$, and $\{W^{ij}\}_{j \in \mathcal{N}_i^a}$ are voltages of buses adjacent  to agent $i$. We can then reformulate problem~\eqref{eq:p2:convex} as
\begin{subequations} \label{eq:admm}
\begin{align}
    \min_{\bg{x}, \bg{z}} \quad & \sum_{i=1}^{n_s} f_i(\bg{x}^i) \label{eq:admm:cost} \\
    \text{s.t.} \quad 
    & \bg{x}^i \in \mathcal{X}_i, && i = 1, \ldots, n_s \label{eq:admm:c1}\\
    & W^{ij} = z^j, V_j = z^j, &&  i = 1, \ldots, n_s, \, j \in \mathcal{N}_i^a \label{eq:admm:c2}
\end{align}
\end{subequations}
where $\bg{z}$ is the consensus variable used to enforce consensus in \eqref{eq:admm:c2}; 
\begin{multline}
    f_i(\bg{x}^i) = \\
    \sum_{j \in \mathcal{N}_i} \left( \sum_{k \in \mathcal{N}_i} \tilde{B}_{jk} V_k + \sum_{k \in \mathcal{N}_i^a} \tilde{B}_{jk} W^{ik} + b_{s,j}u_j - h_j \right)^2;
\end{multline}
and
\begin{multline}
    \mathcal{X}_i = \bigg\{ \bg{x} =  (\bg{V},\bg{u},\bg{W}):\; u_jV_j \ge V_{0,j}^2, r_{0,j}^2u_j \ge V_j, \\
    \sum_{k \in \mathcal{N}_i} \tilde{B}_{jk} V_k + \sum_{k \in \mathcal{N}_i^a} \tilde{B}_{jk} W^{ik} \le h_j, \;\; j \in \mathcal{N}_i \bigg\}.
\end{multline}

The corresponding augmented Lagrangian of problem \eqref{eq:admm} with penalty parameter $\rho$ is given by
\begin{multline}
    L_\rho(\bg{x}, \bg{z}, \bg{\lambda}, \bg{\mu}) = \sum_{i=1}^{n_s} f_i(\bg{x}^i) \\
    + \sum_{i \in \mathcal{B}} \left( \lambda^i (V_i - z^i) + \frac{\rho}{2} (V_i - z^i)^2 \right) \\
    + \sum_{i = 1}^{n_s} \sum_{j \in \mathcal{N}_i^a} \left( \mu^{ij} (W^{ij} - z^j) + \frac{\rho}{2} ( W^{ij} - z^j )^2 \right).
\end{multline}

The ADMM performs the following iterative updates:
\begin{subequations} \label{eq:update}
    \begin{align}
        \bg{x}_{k+1}^i &= \arg \min_{\bg{x}^i \in \mathcal{X}_i} \bigg\{ f_i(\bg{x}^i)
        + \sum_{j \in \mathcal{B} \cup \mathcal{N}_i} \left( \lambda_k^j V_j + \frac{\rho}{2}(V_j - z_k^j)^2 \right) \nonumber\\
        & \quad + \sum_{j \in \mathcal{N}_i^a} \left( \mu_k^{ij} W^{ij} + \frac{\rho}{2} ( W^{ij} - z_k^j )^2 \right) \bigg\}, \label{eq:update:x} \\
        z_{k+1}^i &= \arg \min \bigg\{
        \sum_{j: i\in\mathcal{N}_j^a} \left( -\mu^{ji}_k z^i + \frac{\rho}{2} (W^{ji}_{k+1} - z^i)^2 \right) \nonumber\\
        & \quad -\lambda_k^i z^i + \frac{\rho}{2}(V_{i,k+1} - z^i)^2 \bigg\}, && \hspace{-1.2in} \forall i \in \mathcal{B} \label{eq:update:z} \\
        \lambda_{k+1}^i &= \lambda_k^i + \rho \left( V_{i,k+1} - z_{k+1}^i \right), && \hspace{-1.2in} \forall i \in \mathcal{B} \label{eq:update:lambda} \\
        \mu_{k+1}^{ij} &= \mu_k^{ij} + \rho \left( W_{k+1}^{ij} - z_{k+1}^j \right), && \hspace{-1.2in} \forall i, \forall j \in \mathcal{N}_i^a \label{eq:update:mu}
    \end{align}
\end{subequations}
Note that the minimization problem \eqref{eq:update:z} admits the following closed-form solution since the objective is an unconstrained quadratic function in $z^i$:
\begin{align} \label{eq:closed-form:z}
    z_{k+1}^i = \frac{\lambda_k^i + \rho V_{i,k+1} + \sum_{j:i \in \mathcal{N}_j^a} \left( \mu_k^{ji} + \rho W_{k+1}^{ji} \right)}{\rho(1 + n_i)},
\end{align}
where $n_i$ is the number of neighbors of agent $i$.

In each iteration of the above ADMM algorithm, each agent solves one optimization problem, shares information with its neighbors, and updates multipliers. We can explicitly write out the local communication in each step
of the distributed algorithm:
\begin{enumerate}
    \item Each agent $i$ receives the multipliers $\{ \mu_k^{ij} \}_{j \in \mathcal{N}_i^a}$ and voltage estimates $\{z_k^j\}_{j \in \mathcal{N}_i^a}$ from its neighbors, solves problem \eqref{eq:update:x}, and broadcasts the resulting voltages $\{W_{k+1}^{ij}\}_{j \in \mathcal{N}_i^a}$ to its neighbors.
    \item Each agent $i$ uses its updated voltages $\{V_{\ell,k+1}\}_{\ell \in \mathcal{N}_i \cap \mathcal{B}}$, multipliers $\{\lambda_k^\ell\}_{\ell \in \mathcal{N}_i \cap \mathcal{B}}$, and received bus voltages $\{W^{j\ell}_{k+1}\}^{\ell \in \mathcal{N}_i \cap \mathcal{B}}_{j: \ell \in \mathcal{N}_j^a}$ and multipliers $\{\mu^{j\ell}_k\}^{\ell \in \mathcal{N}_i \cap \mathcal{B}}_{j: \ell \in \mathcal{N}_j^a}$ to compute $\{z_{k+1}^\ell\}_{\ell \in \mathcal{N}_i \cap \mathcal{B}}$ and broadcasts them to its neighbors.
    \item Each agent $i$ updates its multipliers using its own updated bus voltages $\{V_{j,k+1}\}_{j \in \mathcal{N}_i \cap \mathcal{B}}$, $\{W^{ij}_{k+1}\}_{j \in \mathcal{N}_i \cap \mathcal{B}}$, estimated voltages $\{z^j_{k+1}\}_{j \in \mathcal{N}_i \cap \mathcal{B}}$, as well as received voltage estimates $\{z_{k+1}^j\}_{j \in \mathcal{N}_i^a}$.
\end{enumerate}

\section{Simulation Results} \label{sect:sim}

This section demonstrates the effectiveness of the proposed optimization formulations for ROA characterization as well as stability monitoring and reactive power support through numerical simulations on IEEE 39-bus system \cite{bills1970line}. We also examine the proposed methods on more general system models (discrete-time LTC model with constant step-size and deadband, and full power flow model).

Throughout the simulations, base load admittance is assumed to be such that the load power under rated secondary voltage matches the specified base power. The reference secondary voltages are set to be 1 p.u. for all load buses. For simulations using the reactive power model in Sections \ref{sect:sim:ROA} and \ref{sect:sim:admm}, line reactance values are retained while transformers, line resistance and charging capacitance are ignored. The reactive power loads are scaled up by $280\%$ to emulate stressed system condition.

Nonconvex problems are solved by IPOPT v0.5.4 \cite{wachter2006implementation} with MUMPS linear solver. The convex ADMM algorithm is solved using MOSEK v9.1.4 \cite{mosek} with CVX \cite{cvx} interfaced through MATLAB. All computations were done on a laptop with 2.2 GHz 6-Core Intel Core i7 processors and 16GB of memory.

\subsection{Characterization of ROA} \label{sect:sim:ROA}

In this section, we demonstrate the characterization of the ROA on IEEE 39-bus system using \eqref{eq:p1:original}. We examine the impact of system contingency on the ROA characterization.

Fig. \ref{fig:roa} shows the characterized ROAs of the 39-bus system projected on bus 3 and 4 before and after line (8,9) is tripped. The projected ROAs are characterized by solving three instances of \eqref{eq:p1:original} with cost vector set to $\bg{e}^3$, $\bg{e}^4$, and $\bg{e}^3 + \bg{e}^4$ (where $\bg{e}^i$ is the $i$th canonical basis in $\mathbb{R}^n$) under each scenario. Each optimal tap ratio $\bg{r}^i, i = 1, 2, 3$ provides a distinct inner approximation $\mathcal{A}(\bg{r}^i)$ of the true ROA. It is clear that the union $\bigcup_i \mathcal{A}(\bg{r}^i)$ is also an ROA. In Fig. \ref{fig:roa}, the blue shaded region enclosed by black solid lines shows the ROA $\mathcal{Q}^{\mathrm{pre}}$ before line tripping, while the green shaded region enclosed by red dashed lines shows the ROA $\mathcal{Q}^{\mathrm{post}}$ after line tripping. As expected, the projected ROA shrinks significantly after the contingency. 

We take a point that lies inside $\mathcal{Q}^{\mathrm{pre}}$ but outside $\mathcal{Q}^{\mathrm{post}}$ and examine its dynamics before and after line (8,9) is tripped. The chosen point is denoted by $\bg{r}^*$ and is shown in Fig. \ref{fig:roa}. The detailed tap values of $\bg{r}^*$ are given in Table \ref{tb:tap} in Appendix \ref{app:tap}. The evolution of tap position dynamics at bus 8 is shown in Fig. \ref{fig:dynamics_roa}. It is seen that the system collapses after the line is tripped but is stable otherwise.

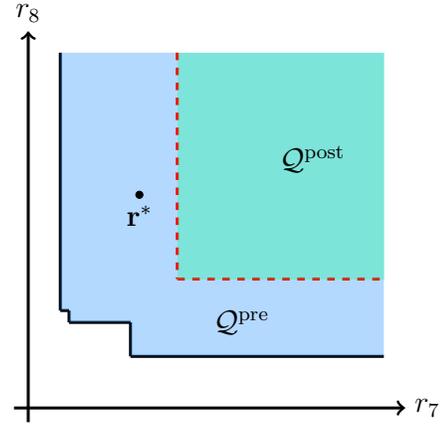
\begin{figure}[!t]
	\centering
	\resizebox{2.4in}{!}{
	\begin{tikzpicture}
		
	\draw[very thick,->] (-0.2,0) -- (5.3,0) node[right] {\large $r_7$};
	\draw[very thick,->] (0,-0.2) -- (0,5.3) node[above] {\large $r_8$};
	
	\draw[very thick] (0.451,5) -- (0.451, 1.371);
	\draw[very thick] (0.451, 1.371) -- (0.574, 1.371);
	\draw[very thick] (0.574, 1.371) -- (0.574, 1.204);
	\draw[very thick] (0.574, 1.204) -- (1.438, 1.204);
	\draw[very thick] (1.438, 1.204) -- (1.438, 0.726);
	\draw[very thick] (1.438, 0.726) -- (5, 0.726);
	\fill[blue!50!cyan,opacity=0.3] (0.451,5) -- (0.451, 1.371) -- (0.574, 1.371) -- (0.574, 1.204) -- (1.438, 1.204) -- (1.438, 0.726) -- (5, 0.726) -- (5,5) -- cycle;
	\node at (3,1.2) {\large $\mathcal{Q}^{\mathrm{pre}}$};
	
	\draw[very thick,dashed,red] (2.096, 5) -- (2.096, 1.815);
	\draw[very thick,dashed,red] (2.096, 1.815) -- (5,1.815);
	\fill[green!50!cyan,opacity=0.3] (2.096, 5) -- (2.096, 1.815) -- (5,1.815) -- (5,5) -- cycle;
	\node at (4,3.5) {\large $\mathcal{Q}^{\mathrm{post}}$};
	
	\draw[fill] (1.563,3) circle [radius=0.05];
	\node at (1.563,2.7) {\large $\bg{r}^*$};
		
	\end{tikzpicture}
	}	
	\caption{ROA Characterizations for IEEE 39-bus system before and after contingency. The blue shaded region enclosed by black solid lines shows the ROA $\mathcal{Q}^{\mathrm{pre}}$ before line tripping, while the green shaded region enclosed by red dashed lines shows the ROA $\mathcal{Q}^{\mathrm{post}}$ after line tripping. The black dot $\bg{r}^*$ shows a point lying outside $\mathcal{Q}^{\mathrm{post}}$ but inside $\mathcal{Q}^{\mathrm{pre}}$.} 
	\label{fig:roa}
\end{figure}

\begin{figure}[!t]
	\centering
	\includegraphics[width=2.8in]{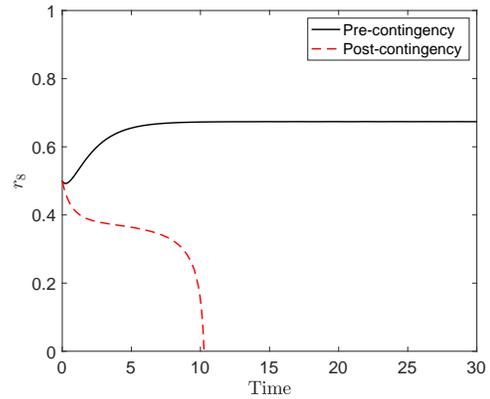}
	\caption{Comparison of tap changer dynamics at bus 8 before and after contingency. Initial tap position is marked by black dot in Fig. \ref{fig:roa}.}
	\label{fig:dynamics_roa}
\end{figure}

\subsection{Distributed Stability Monitoring and Control} \label{sect:sim:admm}

We next examine the performance of the distributed algorithm for stability monitoring and control introduced in Sections \ref{sect:stability_formulation} and \ref{sect:admm}. We test the ADMM algorithm on the following four scenarios: 
\begin{enumerate}
    \item Steady-state after line $(8,9)$ outage where the initial tap position is inside $\mathcal{Q}^{\mathrm{post}}$; 
    \item Tap position $\bg{r}^*$ after line $(8,9)$ outage; 
    \item Tap position $\bg{r}^*$ after line $(8,9)$ outage with additional load increase ($300\%$ higher than base load);
    \item Tap position $\bg{r}^*$ after line $(3,4)$ outage with additional load increase ($300\%$ higher than base load).
\end{enumerate}

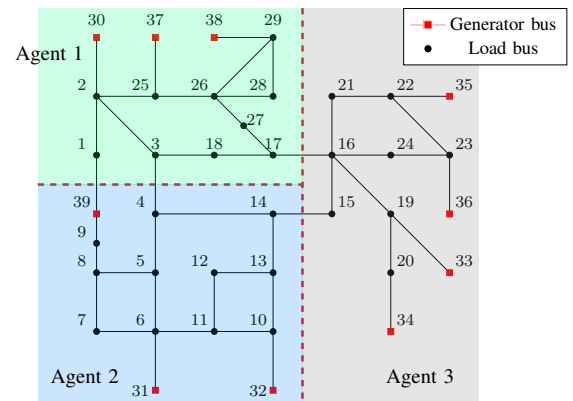
\begin{figure}[!t]
	\centering
	\resizebox{3in}{!}{
	\begin{tikzpicture}
	
	\draw[fill=red,red] (-0.05,6.95) rectangle (0.05,7.05);
	\node at (0,7.3) {\footnotesize $30$};
	\draw[fill=red,red] (0.95,6.95) rectangle (1.05,7.05);
	\node at (1,7.3) {\footnotesize $37$};
	\draw[fill=red,red] (1.95,6.95) rectangle (2.05,7.05);
	\node at (2,7.3) {\footnotesize $38$};
	\draw (2.05,7) -- (3,7);
	\draw[fill] (3,7) circle [radius=0.05];
	\node at (3,7.3) {\footnotesize $29$};
	
	\draw (0,6.95) -- (0,6);
	\draw (1,6.95) -- (1,6);
	\draw (3,7) -- (3,6);
	\draw (3,7) -- (2,6);
	
	\draw[fill] (0,6) circle [radius=0.05];
	\node at (-0.25,6.2) {\footnotesize $2$};
	\draw (0,6) -- (1,6);
	\draw[fill] (1,6) circle [radius=0.05];
	\node at (0.75,6.2) {\footnotesize $25$};
	\draw (1,6) -- (2,6);
	\draw[fill] (2,6) circle [radius=0.05];
	\node at (1.75,6.2) {\footnotesize $26$};
	\draw (2,6) -- (3,6);
	\draw[fill] (3,6) circle [radius=0.05];
	\node at (2.75,6.2) {\footnotesize $28$};
	\draw[fill] (2.5,5.5) circle [radius=0.05];
	\node at (2.7,5.6) {\footnotesize $27$};
	
	\draw (0,6) -- (0,5);
	\draw (0,6) -- (1,5);
	\draw (2,6) -- (3,5);
	
	\draw[fill] (1,5) circle [radius=0.05];
	\node at (1,5.2) {\footnotesize $3$};
	\draw (1,5) -- (2,5);
	\draw[fill] (2,5) circle [radius=0.05];
	\node at (2,5.2) {\footnotesize $18$};
	\draw (2,5) -- (3,5);
	\draw[fill] (3,5) circle [radius=0.05];
	\node at (3,5.2) {\footnotesize $17$};
	
	\draw (1,5) -- (1,4);
	
	\draw[fill] (1,4) circle [radius=0.05];
	\node at (0.75,4.2) {\footnotesize $4$};
	\draw[fill] (3,4) circle [radius=0.05];
	\node at (2.75,4.2) {\footnotesize $14$};
	\draw (1,4) -- (3,4);
	
	\draw (1,4) -- (1,3);
	\draw (3,4) -- (3,3);
	
	\draw[fill] (1,3) circle [radius=0.05];
	\node at (0.75,3.2) {\footnotesize $5$};
	\draw[fill] (2,3) circle [radius=0.05];
	\node at (1.75,3.2) {\footnotesize $12$};
	\draw (2,3) -- (3,3);
	\draw[fill] (3,3) circle [radius=0.05];
	\node at (2.75,3.2) {\footnotesize $13$};
	
	\draw (1,3) -- (1,2);
	\draw (2,3) -- (2,2);
	\draw (3,3) -- (3,2);
	
	\draw[fill] (1,2) circle [radius=0.05];
	\node at (0.75,2.2) {\footnotesize $6$};
	\draw (1,2) -- (2,2);
	\draw[fill] (2,2) circle [radius=0.05];
	\node at (1.75,2.2) {\footnotesize $11$};
	\draw (2,2) -- (3,2);
	\draw[fill] (3,2) circle [radius=0.05];
	\node at (2.75,2.2) {\footnotesize $10$};
	
	\draw (1,2) -- (1,1.05);
	\draw (3,2) -- (3,1.05);
	
	\draw[fill=red,red] (2.95,0.95) rectangle (3.05,1.05);
	\node at (2.75,1) {\footnotesize $32$};
	\draw[fill=red,red] (0.95,0.95) rectangle (1.05,1.05);
	\node at (0.75,1) {\footnotesize $31$};
	
	\draw[fill] (4,6) circle [radius=0.05];
	\node at (4.25,6.2) {\footnotesize $21$};
	\draw (4,6) -- (5,6);
	\draw[fill] (5,6) circle [radius=0.05];
	\node at (5.25,6.2) {\footnotesize $22$};
	\draw (5,6) -- (5.95,6);
	\draw[fill=red,red] (5.95,5.95) rectangle (6.05,6.05);
	\node at (6.25,6.2) {\footnotesize $35$};
	
	\draw (4,6) -- (4,5);
	\draw (5,6) -- (6,5);
	
	\draw (3,5) -- (4,5);
	\draw[fill] (4,5) circle [radius=0.05];
	\node at (4.25,5.2) {\footnotesize $16$};
	\draw (4,5) -- (5,5);
	\draw[fill] (5,5) circle [radius=0.05];
	\node at (5.25,5.2) {\footnotesize $24$};
	\draw (5,5) -- (6,5);
	\draw[fill] (6,5) circle [radius=0.05];
	\node at (6.25,5.2) {\footnotesize $23$};
	
	\draw (4,5) -- (4,4);
	\draw (4,5) -- (5,4);
	\draw (6,5) -- (6,4);
	
	\draw (3,4) -- (4,4);
	\draw[fill] (4,4) circle [radius=0.05];
	\node at (4.25,4.2) {\footnotesize $15$};
	\draw[fill] (5,4) circle [radius=0.05];
	\node at (5.25,4.2) {\footnotesize $19$};
	\draw[fill=red,red] (5.95,3.95) rectangle (6.05,4.05);
	\node at (6.25,4.2) {\footnotesize $36$};
	
	\draw (5,4) -- (5,3);
	\draw (5,4) -- (6,3);
	
	\draw[fill] (5,3) circle [radius=0.05];
	\node at (5.25,3.2) {\footnotesize $20$};
	\draw[fill=red,red] (5.95,2.95) rectangle (6.05,3.05);
	\node at (6.25,3.2) {\footnotesize $33$};
	
	\draw (5,3) -- (5,2.05);
	
	\draw[fill=red,red] (4.95,1.95) rectangle (5.05,2.05);
	\node at (5.25,2.2) {\footnotesize $34$};
    
	\draw[fill] (0,5) circle [radius=0.05];
	\node at (-0.25,5.2) {\footnotesize $1$};
	\draw (0,5) -- (0,4.05);
	\draw[fill=red,red] (-0.05,3.95) rectangle (0.05,4.05);
	\node at (-0.25,4.2) {\footnotesize $39$};
	\draw[fill] (0,3.5) circle [radius=0.05];
	\node at (-0.25,3.7) {\footnotesize $9$};
	\draw[fill] (0,3) circle [radius=0.05];	
	\node at (-0.25,3.2) {\footnotesize $8$};
	\draw[fill] (0,2) circle [radius=0.05];	
	\node at (-0.25,2.2) {\footnotesize $7$};
	\draw (0,3.95) -- (0,2);
	
	\draw (0,3) -- (1,3);
	\draw (0,2) -- (1,2);
	
	\draw[very thick,dashed,red] (3.5,7.5) -- (3.5,0.8);
	\draw[very thick,dashed,red] (-1,4.5) -- (3.5,4.5);
	
	\fill[green!50!cyan,opacity=0.2] (-1,7.5) -- (3.5,7.5) -- (3.5,4.5) -- (-1,4.5) -- cycle;
	\fill[blue!50!cyan,opacity=0.2] (-1,0.8) -- (3.5,0.8) -- (3.5,4.5) -- (-1,4.5) -- cycle;
	\fill[red!50!cyan,opacity=0.2] (3.5,7.5) -- (6.5,7.5) -- (6.5,0.8) -- (3.5,0.8) -- cycle;
	
	\node at (-0.8,6.7) {Agent 1};
	\node at (5.5,1.2) {Agent 3};
	\node at (-0.2,1.2) {Agent 2};
	
	\begin{customlegend}[
    legend entries={ 
    Generator bus,
    Load bus
    },
    legend style={at={(8,7.5)},font=\small}] 
    \addlegendimage{only marks, mark=square*, fill=red, draw=white, sharp plot}
    \addlegendimage{mark=*, ball color=black, draw=white, sharp plot}
    \end{customlegend}
	
	\end{tikzpicture}
	}	
	\caption{Partition of IEEE 39-bus system.} 
	\label{fig:39bus}
\end{figure}

The system is partitioned into three agents as shown in Fig. \ref{fig:39bus}. The algorithm terminates when the objective values settle to relative error (or absolute error when the optimum is $0$) of less than $10^{-4}$ with respect to the global optimum obtained by MOSEK. We set the penalty parameter to be $\rho = 200$. Each variable is initialized to be $0.1$ p.u. larger than their optimal values, which is reasonable considering system operators generally have good knowledge of typical system conditions. The performance of ADMM algorithm is tabulated in Table \ref{tb:admm}. It is seen that the algorithm typically converges in around one hundred iterations, taking tens of seconds for each agent in total. Considering the time-scale of LTC actions in the minute range, the computational complexity is quite reasonable and is well suited for online application.

\begin{table}[!t]
\renewcommand{\arraystretch}{1.1}
\caption{Performance of ADMM Algorithm for Stability Monitoring and Control}
\label{tb:admm}
\centering
\begin{tabular}{ccccc}
\toprule
\multirow{2}{*}{Scenario} & Optimal & \# of & \multirow{2}{*}{Time (sec.)} & Time per \\
& objective & iterations & & subsystem (sec.) \\
\midrule
1 & $0$ & $39$ & $33.01$ & $11.00$ \\
2 & $4.1870$ & $89$ & $73.52$ & $24.51$ \\
3 & $12.3824$ & $83$ & $65.35$ & $21.78$ \\
4 & $20.4829$ & $113$ & $109.72$ & $36.57$ \\
\bottomrule
\end{tabular}
\end{table}

The convergence property of the ADMM solver is shown more clearly in Fig. \ref{fig:iteration}, which presents the evolution of the relative error of the ADMM iterations under scenario 4. Roughly speaking, the relative error is one order of magnitude smaller every 25 iterations or so, and it reduces to less than $1\%$ in less than 50 iterations.


\begin{figure}[!t]
	\centering
	\resizebox{6cm}{!}{
	\begin{tikzpicture}
	\begin{axis}[
		ymode=log,
	    ymin=0.00000001, ymax=10,
		xlabel= Number of iteration,
		ylabel= Relative error]
		\addplot[black,very thick] 
		coordinates{(1, 4.369032073) (2, 4.693635959) (3, 4.075978154) (4, 3.760556166) (5, 3.547373884) (6, 3.282806164) (7, 2.914652463) (8, 2.452570166) (9, 1.941122708) (10, 1.433769105) (11, 0.975838728) (12, 0.597320878) (13, 0.311329321) (14, 0.115612951) (15, 0.003724615) (16, 0.067024670) (17, 0.095550401) (18, 0.106931185) (19, 0.112543005) (20, 0.117145982) (21, 0.120392152) (22, 0.119353673) (23, 0.111072456) (24, 0.094398311) (25, 0.070714238) (26, 0.043590824) (27, 0.017696992) (28, 0.002543552) (29, 0.014113545) (30, 0.016044051) (31, 0.009441081) (32, 0.003025439) (33, 0.017949325) (34, 0.032062235) (35, 0.042929602) (36, 0.049306793) (37, 0.051147813) (38, 0.049323602) (39, 0.045171324) (40, 0.040076099) (41, 0.035094575) (42, 0.030806716) (43, 0.027317649) (44, 0.024395939) (45, 0.021673240) (46, 0.018834547) (47, 0.015744002) (48, 0.012483049) (49, 0.009307736) (50, 0.006552974) (51, 0.004522421) (52, 0.003398588) (53, 0.003196684) (54, 0.003768937) (55, 0.004849549) (56, 0.006124048) (57, 0.007299585) (58, 0.008158800) (59, 0.008587185) (60, 0.008571878) (61, 0.008178352) (62, 0.007515158) (63, 0.006698783) (64, 0.005826847) (65, 0.004964246) (66, 0.004143178) (67, 0.003372203) (68, 0.002649712) (69, 0.001976092) (70, 0.001361341) (71, 0.000826089) (72, 0.000396865) (73, 0.000098428) (74, 0.000053973) (75, 0.000060395) (76, 0.000061995) (77, 0.000281407) (78, 0.000556837) (79, 0.000844727) (80, 0.001106295) (81, 0.001311793) (82, 0.001443652) (83, 0.001495990) (84, 0.001473025) (85, 0.001385498) (86, 0.001247930) (87, 0.001075481) (88, 0.000882383) (89, 0.000681119) (90, 0.000482488) (91, 0.000295499) (92, 0.000128332) (93, 0.000012177) (94, 0.000120583) (95, 0.000193090) (96, 0.000228333) (97, 0.000227588) (98, 0.000195029) (99, 0.000136957) (100, 0.000061689) (101, 0.000021979) (102, 0.000105176) (103, 0.000180274) (104, 0.000241356) (105, 0.000284372) (106, 0.000307692) (107, 0.000311341) (108, 0.000297095) (109, 0.000267687) (110, 0.000226737) (111, 0.000178032) (112, 0.000125694) (113, 0.000073306) (114, 0.000024263) (115, 0.000018614) (116, 0.000053170) (117, 0.000078002) (118, 0.000092350) (119, 0.000096396) (120, 0.000091004) (121, 0.000077634) (122, 0.000058255) (123, 0.000035017) (124, 0.000010233) (125, 0.000013879) (126, 0.000035434) (127, 0.000053085) (128, 0.000065756) (129, 0.000072960) (130, 0.000074806) (131, 0.000071548) (132, 0.000064028) (133, 0.000053183) (134, 0.000040078) (135, 0.000025896) (136, 0.000011821) (137, 0.000001298) (138, 0.000012606) (139, 0.000021461) (140, 0.000027526) (141, 0.000030655) (142, 0.000030929) (143, 0.000028663) (144, 0.000024304) (145, 0.000018405) (146, 0.000011609) (147, 0.000004460) (148, 0.000002425) (149, 0.000008476) (150, 0.000013398) (151, 0.000016921) (152, 0.000018921) (153, 0.000019374) (154, 0.000018462) (155, 0.000016367) (156, 0.000013358) (157, 0.000009724) (158, 0.000005838) (159, 0.000001931) (160, 0.000001646) (161, 0.000004612) (162, 0.000006948) (163, 0.000008442) (164, 0.000009191) (165, 0.000009118) (166, 0.000008363) (167, 0.000007023) (168, 0.000005253) (169, 0.000003351) (170, 0.000001297) (171, 0.000000617) (172, 0.000002317) (173, 0.000003616) (174, 0.000004598) (175, 0.000005144) (176, 0.000005158) (177, 0.000004982) (178, 0.000004340) (179, 0.000003467) (180, 0.000002448) (181, 0.000001332) (182, 0.000000282) (183, 0.000000665) (184, 0.000001514) (185, 0.000002061) (186, 0.000002530) (187, 0.000002696) (188, 0.000002584) (189, 0.000002340) (190, 0.000001937) (191, 0.000001505) (192, 0.000000894) (193, 0.000000329) (194, 0.000000156) (195, 0.000000660) (196, 0.000001055) (197, 0.000001299) (198, 0.000001380) (199, 0.000001445) (200, 0.000001264)
};
	\end{axis}
	\end{tikzpicture}}
	\caption{Relative error of objective value as a function of number of iterations for distributed ADMM solver of \eqref{eq:p2:convex} under scenario 4.}
	\label{fig:iteration}
\end{figure}
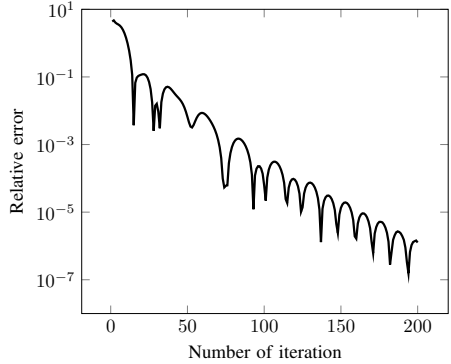

Table \ref{tb:Qsup} shows the percentage of system-wide reactive power reduction in p.u. to restore stability. The reactive powers refers to the power consumption at the equilibrium, i.e., when the secondary voltages are 1 p.u.. Around $5\%$ of load need to be reduced to restore stability for the simulated scenarios. In addition, the reactive power support at individual load level is shown in Fig. \ref{fig:Qsup:s3} for scenario 3.

\begin{table}[!t]
\renewcommand{\arraystretch}{1.1}
\centering
\caption{Reactive Power Support to Restore System Stability}
\begin{tabular}{cccc}
\toprule
Scenario & Total load & Total support & Percentage \\
\midrule
1 & $55.10$ & $0$ & $0\%$ \\
2 & $55.10$ & $1.93$ & $3.50\%$  \\
3 & $58.004$ & $3.31$ & $5.70\%$ \\
4 & $58.004$ & $4.68$ & $8.07\%$ \\
\bottomrule
\end{tabular}
\label{tb:Qsup}
\end{table}

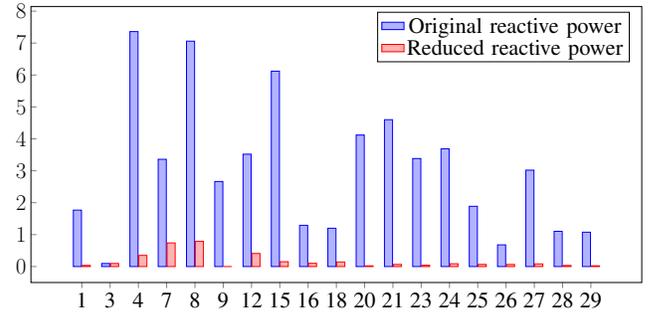
\begin{figure}[!t]
	\centering
	\resizebox{8.5cm}{!}{
	\begin{tikzpicture}
	\begin{axis}[
	ybar,
	bar width=.22cm,
	width=\textwidth,
	height=.5\textwidth,
	tick label style={font=\LARGE},
	tickpos=left,
	xticklabels={1,3,4,7,8,9,12,15,16,18,20,21,23,24,25,26,27,28,29},
	xtick={1,2,3,4,5,6,7,8,9,10,11,12,13,14,15,16,17,18,19},
	ymin=-0.5,
	legend entries={\LARGE{Original reactive power}, \LARGE{Reduced reactive power}},
	y tick label style={/pgf/number format/.cd,%
		scaled y ticks = false,
		set thousands separator={},
		fixed
	},
	]
	\addplot +[bar shift=-.12cm, area legend] coordinates {
		(1,1.768) (2,0.096) (3,7.36) (4,3.36) (5,7.06) (6,2.664) (7,3.52) (8,6.12) (9,1.292) (10,1.2) (11,4.12) (12,4.6) (13,3.384) (14,3.688) (15,1.888) (16,0.68) (17, 3.02) (18,1.104) (19,1.076)};
	
	\addplot  +[bar shift=.12cm, area legend]coordinates {
		(1,0.0354) (2,0.096) (3,0.3497) (4,0.7402) (5,0.7901) (6,0) (7,0.4101) (8,0.151) (9,0.1011) (10,0.1444) (11,0.0216) (12,0.0666) (13,0.0395) (14,0.0874) (15,0.0673) (16,0.0655) (17,0.0831) (18,0.0334) (19,0.0233)};
	\end{axis}
	\end{tikzpicture} }
	\caption{Comparison of original reactive power load and load reduction to restore system stability for scenario 3.}
	\label{fig:Qsup:s3}
\end{figure}

\subsection{Extensions to More General Models} \label{sect:sim:extension}

In this section, we study the effectiveness of the proposed formulation for stability monitoring and instability mitigation under more realistic system models. Two sets of computational experiments are performed. First, we rerun the stability simulations in the last section using discrete-time LTC dynamics model with constant step size and and deadband, as opposed to the continuous-time approximation \eqref{eq:dynamics}. Second, we extend the instability evaluation and mitigation formulation \eqref{eq:p2:convex} to a full-fledged power flow model and examine its performance using the discrete-time LTC model.

\subsubsection{Discrete-time LTC dynamics}
In this section, we examine the continuous-time approximation of the LTC dynamics by comparing it with the discrete-time one. The discrete-time model for bus $i \in \mathcal{V}_L$ is given by \cite[Sect. 4.4]{van2007voltage}
\begin{align} \label{eq:discrete}
	    r^{(k+1)}_i = 
	    \begin{cases} 
	    r^{(k)}_i + \Delta r_i, & \text{if } V_{s,i} > V_{0,i} + d_i, \\
	    r^{(k)}_i - \Delta r_i, & \text{if } V_{s,i} < V_{0,i} - d_i, \\
	    r^{(k)}_i, & \text{otherwise},
	    \end{cases}
\end{align}
in contrast to the continuous-time approximation in \eqref{eq:dynbusi}. For the simulation, we set the deadband $d_i = 0.01$ and tap $\Delta r_i = 0.0125$ for all load bus $i$.

The stability behavior of the LTC dynamics described by the discrete-time model \eqref{eq:discrete} and its continuous-time counterpart \eqref{eq:dynbusi} under the four scenarios in the last section are tabulated in Table \ref{tb:discrete}. Two cases are examined in each scenario: for scenarios 2--4, we simulate LTC dynamics with and without the minimum reactive power support needed to restore system stability based on the continuous-time model; since the continuous-time model is stable post-contingency with no reactive power support in scenario 1, we also include pre-contingency condition for stability evaluation in addition to the post-contingency one. 

It is seen from the table that the stability behaviors are consistent for seven out of eight cases. For the inconsistent one (marked in bold in the table), the continuous-time approximation is unstable whereas the discrete-time one is stable, so the approximation yields conservative yet feasible strategy.

\begin{table}[!t]
\renewcommand{\arraystretch}{1.1}
\caption{Comparison of Stability Behavior by Discrete-Time Model and Its Continuous-Time Approximation}
\label{tb:discrete}
\centering
\begin{tabular}{cccc}
\toprule
\multicolumn{2}{c}{\multirow{2}{*}{Case}} & \multicolumn{2}{c}{Stability} \\
& & Continuous & Discrete \\
\midrule
\multirow{2}{*}{Scenario 1} & Pre-contingency & Stable & Stable \\
& Post-contingency & Stable & Stable \\
\multirow{2}{*}{Scenario 2} & With support & Stable & Stable \\
& Without support & \textbf{Unstable} & \textbf{Stable} \\
\multirow{2}{*}{Scenario 3} & With support & Stable & Stable \\
& Without support & Unstable & Unstable \\
\multirow{2}{*}{Scenario 4} & With support & Stable & Stable \\
& Without support & Unstable & Unstable \\
\bottomrule
\end{tabular}
\end{table}

\subsubsection{Full power flow model}

In this section, we examine the performance of the natural extension of the proposed main formulation \eqref{eq:p2:convex} to full power flow model. Let tap position $\bg{r}_0$, reference secondary side voltage $\bg{V}_0$, base load conductance $\bg{g}_s$ and susceptance $\bg{b}_s$, real power generation setpoints $\bg{P}^{\mathrm{sp}}$ and voltage magnitude setpoints $\bg{V}^{\mathrm{sp}}$ for generator buses be given. Denote the power flow equations at bus $i$ by $P_i(\bg{V},\bg{\theta})$ and $Q_i(\bg{V},\bg{\theta})$, 
and the squared secondary side voltage magnitudes by $\bg{u}$, then the instability monitoring and mitigation problem can be formulated as
\begin{subequations} \label{eq:full}
\begin{align}
    \min_{ \substack{\bg{u}, \bg{b}_\mathrm{sup} \\ \bg{V}, \bg{\theta}} } \quad & \sum\nolimits_{i \in \mathcal{V}_L} \left( \frac{|g_{s,i} +jb_{s,i}|}{b_{s,i}} b_{\mathrm{sup},i} \right)^2 \label{eq:full:obj} \\
    \text{s.t.} \quad 
    & P_i(\bg{V},\bg{\theta}) = u_i \frac{g_{s,i}}{b_{s,i}}(b_{s,i} - b_{\mathrm{sup},i}), && i \in \mathcal{V}_L \label{eq:full:b} \\
    & Q_i(\bg{V},\bg{\theta}) = u_i (b_{s,i} - b_{\mathrm{sup},i}), && i \in \mathcal{V}_L \label{eq:full:c} \\
    & P_i(\bg{V},\bg{\theta}) = P_i^{\mathrm{sp}}, && \hspace{-0.37in} i \in \mathcal{V}_G \setminus \mathcal{V}_{\mathrm{ref}} \label{eq:full:d} \\
    & |V_i| = V_i^{\mathrm{sp}}, && i \in \mathcal{V}_G \label{eq:full:e} \\
    & u_i \ge V_{0,i}^2, && i \in \mathcal{V}_L \label{eq:full:f} \\
    & |V_i|^2 \le u_i r_{0,i}^2, && i \in \mathcal{V}_L \label{eq:full:g} \\
    & -b_{s,i}^- \le b_{\mathrm{sup},i} \le b_{s,i}^+, && i \in \mathcal{V}_L \label{eq:full:h}
\end{align}
\end{subequations}
Equations \eqref{eq:full:b}--\eqref{eq:full:c} model the power flow equations of load buses (with LTC), where the loads are modeled by constant admittance models with constant power factor. Equations \eqref{eq:full:d}--\eqref{eq:full:e} fix the real powers and voltage magnitudes of generator buses at their set-points as these quantities are not regulated at the load side. Equations \eqref{eq:full:f}--\eqref{eq:full:g} ensure $\bg{r}_0 \in \mathcal{P}$. Finally, \eqref{eq:full:h} guarantees the load type (generation or consumption) stays the same, where $b_{s,i}^+ = \max\{b_{s,i}, 0\}$ and $b_{s,i}^- = \max\{-b_{s,i}, 0\}$. The problem is formulated as an extended optimal power flow problem in \textsc{Matpower} \cite{zimmerman2011matpower} and is solved by IPOPT.

Tests are carried out on the 39-bus system with $9$ p.u. shunt capacitance added at buses 5, 6, 11, 14, and 17 to keep voltage profile high even under stressed loading condition --- a scenario prone to system instability. We scale all load admittance and generator real powers to $2.89$ times of their base values and calculate the steady-state primary side voltages, which we define to be the initial tap position $\bg{r}_0$. The average load voltage is $0.92$ p.u., with only two voltages below $0.85$ p.u., which looks fairly healthy. We simulate system stability under this loading condition subject to line outage contingencies.

Out of 46 lines in the system, we tested all 35 line outages whose removal does not disconnect the network. All 14 cases deemed stable by the optimization problem \eqref{eq:full} (with zero optimal cost) are verified to be indeed stable by numerical simulations of the discrete-time dynamics \eqref{eq:discrete}. For the 21 cases whose stability behaviors are undecided by \eqref{eq:full:obj}, 10 of them turn out to be unstable. For the 21 indefinite cases, minimum load side support determined by \eqref{eq:full} are provided and their stability are reevaluated. Three cases are still unstable even after applying the support (marked by bold in the table), all of which are unstable before load side support. This suggests the proposed stability condition can be inexact for full power flow model. One can derive more conservative load side support strategies to ensure system stability. Simulation results show that stability can be restored for all three cases by solving \eqref{eq:full} with $r_{0,i}' = r_{0,i} - 0.1$, i.e., by forcing stability condition \eqref{eq:full:g} to hold for more conservative tap positions. Detailed simulation results are summarized in Table \ref{tb:full} in Appendix \ref{app:tap}.

In this subsection, we have briefly examined the generalizability of the proposed approach to discrete-time dynamics and full power flow model with promising results, but a thorough and rigorous analysis is subject to future work as it requires significantly different mathematical development.

\section{Conclusion}


In this paper, we have studied the monitoring and control of long-term voltage stability considering LTC dynamics. It has been shown that for reactive power model, the networked LTC dynamics admit a unique stable equilibrium genericsdally and a large ROA can be characterized in a computationally efficient fashion. An ADMM-based distributed algorithm leveraging the developed ROA characterization is capable of monitoring system voltage stability and computing secondary reactive power support in emergency conditions. Simulation results suggest the extension of the proposed stability monitoring and control approach to more general models is empirically sound. Future research will study such generalization in a more thorough and rigorous way.

\appendices
\section{Some Known Results on Stability of LTC System} \label{app:known}

\begin{thm}[\hspace{1sp}{\cite[Lem. 3]{liu1989analysis}}] \label{thm:invariant}
    $\mathcal{P}$ is an invariant set of the dynamical system \eqref{eq:dynamics}.
\end{thm}

\begin{thm}[\hspace{1sp}{\cite[Prop. 2]{liu1989analysis}}] \label{thm:alpha}
    Assume that the set $\mathcal{P}$ defined in \eqref{eq:P} is non-empty. Then $\mathcal{P}$ has a largest element $\bg{\alpha}$, i.e., $\bg{\alpha} \ge \bg{r}, \forall \bg{r} \in \mathcal{P}$. Furthermore, $\bg{\alpha} \in \mathcal{M}$.
\end{thm}

\begin{thm}[\hspace{1sp}{\cite[Prop. 3]{liu1989analysis}}]
    If the set 
    \begin{multline} \label{eq:Pa}
        \mathcal{P}_{\bg{\alpha}} = \{ \bg{r} :\; \bg{r} \in \mathcal{P} \setminus \{\bg{\alpha}\}, \text{ and there is no equilibrium } \bg{e} \\
        \text{with } \bg{r} \le \bg{e} < \bg{\alpha} \}      
    \end{multline}
    is non-empty then $\bg{\alpha}$ is asymptotically stable and the union
    \begin{align}
        \mathcal{A} = \bigcup_{\ubar{\bg{r}} \in \mathcal{P}_{\bg{\alpha}}} \{\bg{r} :\; \bg{r} \ge \ubar{\bg{r}} \}
    \end{align}
    is an ROA of $\bg{\alpha}$.
\end{thm}

\begin{thm}[\hspace{1sp}{\cite[Prop. 4]{liu1989analysis}}] \label{thm:nonempty}
    If the Jacobian matrix of the dynamical system \eqref{eq:dynamics} is nonsingular at $\bg{\alpha}$ then the set $\mathcal{P}_{\bg{\alpha}}$ as defined by \eqref{eq:Pa} is nonempty.
\end{thm}

\section{Supporting Lemmas}
We first present a slightly generalized version of the result concerning uniqueness of complex fixed point in \cite{Cui2019solvability}:
\begin{lem}[\hspace{1sp}{\cite[Thm. C.4]{Cui2019solvability}}] \label{thm:complexunique}
    Given $X = \Pi_{i=1}^n X_i$ where each $X_i$ is a non-empty simply connected open proper subset of $\mathbb{C}$. Let $f : \cl(X) \to X$ be a function holomorphic in $X$ and continuous on $\cl(X)$, and $\cl(f(X))$ is contained in $X$. Then $f$ has exactly one fixed point in $X$. Moreover, the sequence $\{ \bg{z}^n \}$ defined as $\bg{z}^{n+1} = f(\bg{z}^n), n = 0, 1, 2, \ldots$ converges to the unique fixed point $\bg{w}$ given any $\bg{z}^0 \in X$.
\end{lem}

\begin{lem} \label{thm:uniquelem}
    Given $X = \Pi_{i=1}^n X_i$ where $X_i = (a_i,b_i) \subset \mathbb{R}_{>0}$. Let parameters $c_i, d_{ik} \in \mathbb{R}_{>0}$ for all $i,k \in \{ 1, \ldots, n\}$. Then the mapping $f:\; \cl(X) \to X$ defined as
    \begin{align}
        f_i(\bg{x}) = c_i - \sum_{k=1}^n \frac{d_{ik}}{x_k}, \quad \forall i \in \{1, \ldots, n\}
    \end{align}
    has a unique fixed point in $X$.
\end{lem}
\begin{proof}
    Since $\cl(X)$ is a convex and compact subset of $\mathbb{R}^n$ and $f(\bg{x})$ is continuous on $\cl(X)$, $f(\bg{x})$ has a fixed point in $\cl(X)$ by Brouwer fixed-point theorem \cite[6.3.2]{ortega1970iterative}. Furthermore, since $f(\cl(X)) \subseteq X$, all the fixed points lie in $X$. 
    
    Define 
    \begin{multline}
        Y_i = \bigg\{ y \in \mathbb{C} :\; a_i < \re(y) < b_i, \\
        -\epsilon - \sum_{k=1}^n \frac{d_{ik}}{a_k} < \im(y) < \epsilon + \sum_{k=1}^n \frac{d_{ik}}{a_k} \bigg\},
    \end{multline}
    where $\epsilon$ is some given positive number. The mapping $f$ is holomorphic on $\cl(Y)$ where $Y = \Pi_{i=1}^n Y_i$. In addition, we claim that $f(\cl(Y)) \subseteq Y$. To see this, note that for $\bg{y} \in \cl(Y)$, we have
    \begin{align}
        \re(f_i(\bg{y})) &= \re\left( c_i - \sum_{k=1}^n \frac{d_{ik}}{y_k} \right)
        = c_i - \sum_{k=1}^n \re\left( \frac{d_{ik}}{y_k} \right) \nonumber\\
        &\ge c_i - \sum_{k=1}^n  \frac{d_{ik}}{|y_k|}
        \ge c_i - \sum_{k=1}^n  \frac{d_{ik}}{a_k}
        > a_i
    \end{align}
    and 
    \begin{align}
        \left| \im(f_i(\bg{y})) \right| &= 
        \left| \im\left( c_i - \sum_{k=1}^n \frac{d_{ik}}{y_k} \right) \right|
        = \left| \sum_{k=1}^n \im\left( \frac{d_{ik}}{y_k} \right) \right| \nonumber\\
        &\le \sum_{k=1}^n \left| \im\left( \frac{d_{ik}}{y_k} \right)\right|
        \le \sum_{k=1}^n \left| \frac{d_{ik}}{y_k} \right|
        \le \sum_{k=1}^n \frac{d_{ik}}{a_k} \nonumber\\
        &< \epsilon + \sum_{k=1}^n \frac{d_{ik}}{a_k}.
    \end{align}
    Since $Y_i, i = 1, \ldots, n$ are simply connected (since they are convex) open proper subsets of $\mathbb{C}$, Lemma \ref{thm:complexunique} ensures that there exists a unique fixed point in $Y$. 
    
    Since $X \subset Y$ and there is at least one fixed point in $X$, there is exactly one fixed point in $X$.
\end{proof}

\begin{lem} \label{thm:connectd}
    Suppose $\{\bg{\alpha}\}$ is a proper subset of $\mathcal{P}$, then for any $\bg{r} \in \mathcal{P}$ and any open proper superset $\mathcal{B}$ of $\bg{r}$, we have $\{\bg{r}\} \subset \mathcal{B} \cap \mathcal{P}$.
\end{lem}
\begin{proof}
    It suffices to show $\mathcal{P}$ is connected. To see this, suppose an open proper superset $\mathcal{B}$ of $\bg{r}$ exists such that $\mathcal{B} \cap \mathcal{P} = \{\bg{r}\}$, it then follows that $\{\bg{r}\}$ and $\mathcal{P} \setminus \{\bg{r}\}$ are separated, a contradiction. 
    
    Given $\bg{s}, \bg{t} \in \mathcal{P}$, we want to show the existence of a path from $\bg{s}$ to $\bg{t}$ defined by a continuous mapping $f:[0,1] \to \mathcal{P}$ with $f(0) = \bg{s}$ and $f(1) = \bg{t}$. Define the convex set 
    \begin{align}
        \mathcal{S}\! =\! \left\{ (\bg{V},\bg{u}):\; \bg{B}_{LL} \bg{V} + [\bg{b}_s]\bg{u} = \bg{h}, [\bg{u}]\bg{V} \ge [\bg{V}_0]\bg{V}_0 \right\}.
    \end{align}
    Let $\bg{V}_s$ and $\bg{V}_t$ be the load voltages corresponding to $\bg{s}$ and $\bg{t}$, respectively. Then it is easy to verify that both $(\bg{V}_s, [\bg{s}]^{-2}\bg{V}_s)$ and $(\bg{V}_t, [\bg{t}]^{-2}\bg{V}_t)$ are in $\mathcal{S}$. In addition, $\sqrt{[\bg{u}]^{-1}\bg{V}} \in \mathcal{P}$ for any $(\bg{V},\bg{u}) \in \mathcal{S}$. 
    Since $\mathcal{S}$ is convex, $(\bg{V}(\lambda),\! \bg{u}(\lambda))\!:=\! \left( (1-\lambda)\bg{V}_s \!+\! \lambda \bg{V}_t, (1-\lambda)[\bg{s}]^{-2}\bg{V}_s + \lambda [\bg{t}]^{-2}\bg{V}_t \right) \in \mathcal{S}$ for $0 \le \lambda \le 1$. The function $f(\lambda)\! =\! \sqrt{[\bg{u}(\lambda)]^{-1}\bg{V}(\lambda)}$ is as desired.
\end{proof}
    

\section{Proof of Lemma \ref{thm:uniqueeq}} \label{app:uniqueeq}
\begin{proof}
    We left multiply $[\bg{V}_L]^{-1}$ on both sides of \eqref{eq:pf} ($[\bg{V}_L]^{-1}$ is well-defined since $\bg{r}, \bg{V}_0 > \mathbbold{0}$), move $\bg{B}_{LG}\bg{V}_G$ to the right, and left multiply $\bg{Z}$ on both sides to yield
    \begin{align} \label{eq:fppf:Vp}
        V_i(\bg{r}) = E_i - \sum_{k \in \mathcal{V}_L} Z_{ik} \frac{Q_k(\bg{r})}{V_k(\bg{r})}, \quad \forall i \in \mathcal{V}_L,
    \end{align}
    where both $\bg{E} = -\tilde{\bg{B}}_{LL}^{-1}\bg{B}_{LG}\bg{V}_G$ and $\bg{Z} = \tilde{\bg{B}}_{LL}^{-1}$ are independent of $\bg{r}$. By replacing $\bg{V}_L$ and $\bg{Q}_L$ with $[\bg{r}]\bg{V}_s$ and $[\bg{V}_s]^2\bg{b}_s$, \eqref{eq:fppf:Vp} can be rewritten as
    \begin{align} \label{eq:fppf:intmed}
        V_{s,i}(\bg{r})r_i = E_i - \sum_{k \in \mathcal{V}_L} Z_{ik} \frac{V_{s,k}(\bg{r})b_k}{r_k}, \quad \forall i \in \mathcal{V}_L.
    \end{align}
    Dividing both sides by $V_{s,i}$, we get a fixed-point form of $\bg{r}$:
    \begin{align} \label{eq:fppf:r}
        r_i = \frac{1}{V_{s,i}(\bg{r})} \left( E_i - \sum_{k \in \mathcal{V}_L} Z_{ik} \frac{V_{s,k}(\bg{r})b_k}{r_k} \right), \quad \forall i \in \mathcal{V}_L.
    \end{align}
    Note that any $\bg{r} \in \mathbb{R}_{>0}^n$ is a fixed point of \eqref{eq:fppf:r}. Define the function $f(\bg{r}) :\; \mathbb{R}^n_{>0} \to \mathbb{R}^n$ as
    \begin{align}
        f_i(\bg{r}) := \frac{1}{V_{0,i}} \left( E_i - \sum_{k \in \mathcal{V}_L} Z_{ik} \frac{V_{0,k} b_k}{r_k} \right), \quad \forall i \in \mathcal{V}_L
    \end{align}
    then the fixed-point mapping \eqref{eq:fppf:r}, when $\bg{V}_s(\bg{r}) = \bg{V}_0$, can be rewritten in a compact manner as
    \begin{align} \label{eq:fppf:compact}
        r_i = f_i(\bg{r}), \quad \forall i \in \mathcal{V}_L.
    \end{align}
    The fixed points of \eqref{eq:fppf:compact} correspond to the equilibria of \eqref{eq:dynamics}, and our goal is to show that there is a unique fixed point to \eqref{eq:fppf:compact} that lies in $\mathcal{A}(\bg{r})$ for any $\bg{r} \in \mathcal{P}$.
    
    Let $\ubar{\bg{r}} \in \mathcal{P}$ be given. If $\ubar{\bg{r}} \notin \mathcal{M}$, there exists some $\ell \in \mathcal{V}_L$ such that $V_{s,\ell}(\ubar{\bg{r}}) > V_{0,\ell}$. The RHS of \eqref{eq:fppf:r} is strictly decreasing with respect to $\bg{V}_s$ since 1) $1/V_{s,i}$ strictly decreases with $V_{s,i}$; 2) $E_i - \sum_{k \in \mathcal{V}_L} \left( Z_{ik} V_{s,k}(\bg{r}) b_k/r_k \right)$ strictly decreases with $\bg{V}_s$ since each $Z_{ik} b_k/r_k$ term is strictly positive; 3) both $1/V_{s,i}$ and $E_i - \sum_{k \in \mathcal{V}_L} \left( Z_{ik} V_{s,k}(\bg{r}) b_k/r_k \right)$ are positive. Since $\bg{V}_s(\ubar{\bg{r}}) \ge \bg{V}_0$ with at least one strict inequality and $\ubar{\bg{r}}$ is a fixed point of \eqref{eq:fppf:r}, we have
    \begin{align}
        \ubar{r}_i < \frac{1}{V_{0,i}} \left( E_i - \sum_{k \in \mathcal{V}_L} Z_{ik} \frac{V_{0,k} b_k}{\ubar{r}_k} \right), \quad \forall i \in \mathcal{V}_L.
    \end{align}
    We define the following set 
    \begin{align}
        \mathcal{I}(\bg{r}) = \left\{ \bg{x} \in \mathbb{R}^n_{> 0} :\; \ubar{r}_i < x_i < r_i, \; \forall i \in \mathcal{V}_L \right\}.
    \end{align}
    The facts that $f_i(\bg{r})$ is a strictly increasing function in $\mathbb{R}_{>0}^n$ which is upper bounded by $E_i/V_{0,i}$ and $\ubar{r}_i < f_i(\ubar{\bg{r}})$ imply $\ubar{r}_i < f_i(\bg{r}) < E_i/V_{0,i}$ as long as $\bg{r} \ge \ubar{\bg{r}}$, which further implies $\ubar{r}_i < f_i(\bg{r}) < \bar{r}_i$ as long as $\bar{\bg{r}} \ge [\bg{V}_0]^{-1}\bg{E}$. In other words, we have $f(\cl(\mathcal{I}(\bg{r}))) \subset \mathcal{I}(\bg{r})$ as long as $\bg{r} \ge [\bg{V}_0]^{-1}\bg{E}$. Lemma \ref{thm:uniquelem} can then be invoked which ensures that there is a unique fixed point in $\mathcal{I}(\bg{r})$ for all $\bg{r} \ge [\bg{V}_0]^{-1}\bg{E}$. This shows there is a unique fixed point in $\inte(\mathcal{A}(\ubar{\bg{r}}))$. In addition, there can not be any equilibrium on $\partial \mathcal{A}(\ubar{\bg{r}})$ since $f$ is increasing in $\mathbb{R}^n_{>0}$ and $\ubar{\bg{r}} < f(\ubar{\bg{r}})$. Therefore, there is a unique equilibrium in $\mathcal{A}(\ubar{\bg{r}})$.
    
    Now we assume $\ubar{\bg{r}} \in \mathcal{M}$. Based on similar argument as above, there is no equilibrium in $\partial \mathcal{A}(\ubar{\bg{r}})$, so the distance $d = \min_{\bg{r} \in \mathcal{M} \setminus\{\ubar{\bg{r}}\}} \min_{i \in \{1,\ldots,n\}} |r_i - \ubar{r}_i |$ is positive. Lemma \ref{thm:connectd} ensures that for any given $\epsilon \in (0,d)$, there exists a point $\tilde{\bg{r}}$ distinct from $\ubar{\bg{r}}$ such that $\tilde{\bg{r}} \in \mathcal{P} \cap \mathcal{B}_\infty(\ubar{\bg{r}},\epsilon)$. Leveraging the result for the case when $\ubar{\bg{r}} \notin \mathcal{M}$ above, we know there is a unique equilibrium in $\mathcal{A}(\tilde{\bg{r}})$. To show there is a unique equilibrium in $\mathcal{A}(\ubar{\bg{r}})$ other than $\ubar{\bg{r}}$, it only remains to show there are no equilibrium in the set $\mathcal{A}(\ubar{\bg{r}}) \setminus \mathcal{A}(\tilde{\bg{r}})$ other than $\ubar{\bg{r}}$, where the set can be represented as
    \begin{align}
        \mathcal{A}(\ubar{\bg{r}}) \setminus \mathcal{A}(\tilde{\bg{r}}) = \{ \bg{r} \ge \ubar{\bg{r}} :\; r_i < \tilde{r}_i \text{ for some } i \in \mathcal{V}_L \}.
    \end{align}
    Suppose the set is non-empty, then for every $\bg{r} \in \mathcal{A}(\ubar{\bg{r}}) \setminus \mathcal{A}(\tilde{\bg{r}})$ distinct from $\ubar{\bg{r}}$, there is an index $i$ such that $\ubar{r}_i \le r_i < \tilde{r}_i$, which means $|r_i - \ubar{r}_i| < |\tilde{r}_i - \ubar{r}_i| < \epsilon < d$, so $\bg{r} \notin \mathcal{M}$.
\end{proof}

\section{Proof of Theorem \ref{thm:eq_unstable}} \label{app:eq_unstable}
\begin{proof}
    
    Take an equilibrium $\bg{r}^* \in \mathcal{M}$ distinct from $\bg{\alpha}$. Since the points in $\mathcal{M}$ are finite and isolated, the minimum distance between $\bg{r}^*$ and other equilibria is positive. Denote the minimum distance by $d := \min_{\bg{r} \in \mathcal{M} \setminus \{\bg{r}^*\}} \| \bg{r}^* - \bg{r} \|_2$, then we know from Lemma \ref{thm:connectd} that for any open ball $\mathcal{B}_2(\bg{r}^*, \epsilon)$ with $\epsilon \in (0,d)$, there is a point $\tilde{\bg{r}} \in \mathcal{B}_2(\bg{r}^*, \epsilon) \cap \mathcal{P}$ which is not an equilibrium. 
    We claim that $\tilde{r}_i > r^*_i$ for at least one $i$. Assume on the contrary that $\tilde{\bg{r}} \le \bg{r}^*$, then it follows from Lemma \ref{thm:uniqueeq} that there is a unique equilibrium in $\mathcal{A}(\tilde{\bg{r}})$, contradicting the fact that both $\bg{r}^*$ and $\bg{\alpha}$ are in $\mathcal{A}(\tilde{\bg{r}})$, so the claim is proven. Since $\mathcal{P}$ is invariant (Theorem \ref{thm:invariant}), bounded above by $\bg{\alpha}$ (Theorem \ref{thm:alpha}), and $\dot{\bg{r}} \ge \mathbbold{0}$ for $\bg{r} \in \mathcal{P}$, the trajectory through $\tilde{\bg{r}}$ converges to an equilibrium. Since $\tilde{r}_i > r^*_i$ for at least one $i$, it converges increasingly to some equilibrium other than $\bg{r}^*$, with distance at least $d$ from $\bg{r}^*$. Since this holds for any $\epsilon \in (0,d)$, the equilibrium $\bg{r}^*$ is unstable.
\end{proof}

\section{Additional Simulation Results}
\label{app:tap}

\begin{table}[!h]
\renewcommand{\arraystretch}{1.1}
\caption{Tap Ratios of the Point in Figure \ref{fig:roa}}
\label{tb:tap}
\centering
\begin{tabular}{ccccccccc}
\toprule
Bus No. & 1 & 3 & 4 & 7 & 8 & 9 & 12 \\
\hline
Tap ratio & 0.83 & 0.64 & 0.45 & 0.36 & 0.50 & 0.50 & 0.38 \\
\hline
Bus No. & 15 & 16 & 18 & 20 & 21 & 23 & 24 \\
\hline
Tap ratio & 0.53 & 0.62 & 0.62 & 0.79 & 0.67 & 0.77 & 0.62 \\
\hline
Bus No. & 25 & 26 & 27 & 28 & 29 \\
\hline
Tap ratio & 0.79 & 0.71 & 0.65 & 0.78 & 0.82 \\
\bottomrule
\end{tabular}
\end{table}

\begin{table}[!t]
\centering
\caption{Simulation Results on Performance of \eqref{eq:full} for Stability Evaluation and Instability Mitigation}
\renewcommand{\arraystretch}{1.1}
\begin{tabular}{cccc}
\toprule
Line outage & Stability certificate & Stability & Stab. w/ min. ctrl. \\
\midrule
(1, 2) & Indefinite & Yes & Yes \\
(1, 39) & Stable & Yes & --- \\
(2, 3) & Indefinite & No & \textbf{No} \\
(2, 25) & Indefinite & Yes & Yes \\
(3, 4) & Stable & Yes & --- \\
(3, 18) & Stable & Yes & --- \\
(4, 5) & Indefinite & No & Yes \\
(4, 14) & Stable & Yes & --- \\
(5, 6) & Indefinite & Yes & Yes \\
(5, 8) & Stable & Yes & --- \\
(6, 7) & Indefinite & Yes & Yes \\
(6, 11) & Indefinite & Yes & Yes \\
(7, 8) & Stable & Yes & --- \\
(8, 9) & Indefinite & Yes & Yes \\
(9, 39) & Stable & Yes & --- \\
(10, 11) & Stable & Yes & --- \\
(10, 13) & Indefinite & No & Yes \\
(12, 11) & Stable & Yes & --- \\
(12, 13) & Stable & Yes & --- \\
(13, 14) & Indefinite & No & Yes \\
(14, 15) & Indefinite & No & Yes \\
(15, 16) & Indefinite & No & \textbf{No} \\
(16, 17) & Indefinite & Yes & Yes \\
(16, 21) & Indefinite & Yes & Yes \\
(16, 24) & Stable & Yes & --- \\
(17, 18) & Stable & Yes & --- \\
(17, 27) & Indefinite & No & Yes \\
(21, 22) & Indefinite & No & Yes \\
(22, 23) & Stable & Yes & --- \\
(23, 24) & Indefinite & Yes & Yes \\
(25, 26) & Indefinite & Yes & Yes \\
(26, 27) & Indefinite & No & \textbf{No} \\
(26, 28) & Stable & Yes & --- \\
(26, 29) & Indefinite & Yes & Yes \\
(28, 29) & Indefinite & No & Yes \\
\bottomrule
\end{tabular} \label{tb:full}
\end{table}

\ifCLASSOPTIONcaptionsoff
  \newpage
\fi



%

\bibliographystyle{IEEEtran}
\bibliography{bibtex/bib/IEEEabrv,bibtex/bib/ref_LTC}




%









\end{document}